\documentclass[11pt, a4paper]{article}
\usepackage{amsmath,amsfonts,amssymb,latexsym,graphics,epsfig,url}
\usepackage{hyperref}
\usepackage{color}
\usepackage{amsthm}
\usepackage{graphicx}
\usepackage{amssymb}
\usepackage[english]{babel}

\newcommand\blfootnote[1]{%
  \begingroup
  \renewcommand\thefootnote{}\footnote{#1}%
  \addtocounter{footnote}{-1}%
  \endgroup
}

\newtheorem{theo}{Theorem}
\newtheorem{lema}[theo]{Lemma}

\newtheorem{propo}[theo]{Proposition}

\theoremstyle{definition}
\newtheorem{example}{Example}
\DeclareMathOperator{\tr}{tr}

\DeclareMathOperator{\spec}{sp}
\newcommand{\qqed}{\hfill$\Box$\medskip}

\def\vec0{\mbox{\boldmath $0$}}

\def\A{\mbox{\boldmath $A$}}

\def\J{\mbox{\boldmath $J$}}

\def\M{\mbox{\boldmath $M$}}
\def\N{\mbox{\boldmath $N$}}

\def\S{\mbox{\boldmath $S$}}

\def\G{\Gamma}
\def\Re{\mathbb R}

\begin{document}

\title{The spectral excess theorem for graphs with few eigenvalues whose distance-$2$ or distance-$1$-or-$2$ graph is strongly regular
\thanks{Research of C. Dalf\'{o} and M. A. Fiol is partially supported by AGAUR under project 2017SGR1087. Research of J. Koolen is partially supported by the {\em National Natural Science Foundation of China} under project No. 11471009,  and the {\em Chinese Academy of Sciences} under its `100 talent' program.}}

\author{C. Dalf\'o$^a$, M.A. Fiol$^{a,b}$, J. Koolen$^c$
\\ \\
{\small $^a$Universitat Polit\`ecnica de Catalunya, BarcelonaTech} \\
{\small Dept. de Matem\`atiques, Barcelona, Catalonia}\\
{\small $^b$ Barcelona Graduate School of Mathematics} \\
{\small {\tt\{cristina.dalfo,miguel.angel.fiol\}@upc.edu}} \\
{\small $^c$University of Science and Technology of China} \\
{\small School of Mathematical Sciences} \\
{\small  Hefei, Anhui, China}\\
{\small {\tt koolen@ustc.edu.cn}}
 }
\date{}

\maketitle

\blfootnote{
\begin{minipage}[l]{0.3\textwidth} \includegraphics[trim=10cm 6cm 10cm 5cm,clip,scale=0.15]{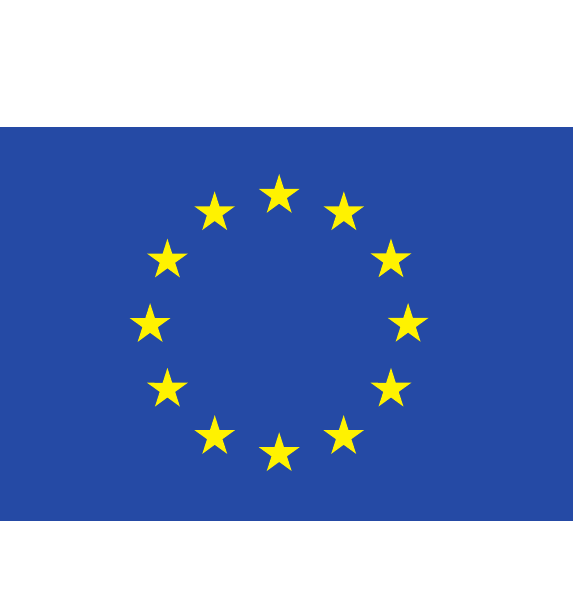} \end{minipage}  \hspace{-2cm} \begin{minipage}[l][1cm]{0.79\textwidth}
   The research of C. Dalf\'{o} has also received funding from the European Union's Horizon 2020 research and innovation programme under the Marie Sk\l{}odowska-Curie grant agreement No 734922.
  \end{minipage}}

\begin{abstract}
We study regular graphs
whose distance-$2$ graph or distance-$1$-or-$2$ graph is strongly regular. We provide a characterization of such graphs $\G$ (among regular graphs with
few distinct eigenvalues) in terms of the spectrum and the mean number of vertices at maximal distance $d$ from every vertex,
where $d+1$ is the number of different eigenvalues of $\G$.
This can be seen as another version of the so-called spectral excess theorem, which characterizes in a similar way those regular graphs that are distance-regular.
\end{abstract}



\section{Preliminaries}

Let $\G$ be a distance-regular graph with adjacency matrix $\A$ and $d+1$ distinct eigenvalues.
The distance-$i$ graph (associated with $\G$) is the graph $\G_i$ having the same vertices as
$\G$ and in which two vertices are adjacent if and only if they are at distance $i$ in $\G$. Similarly,
the distance-$i$-or-$j$ graph is the graph $\G_{i,j}$ with the same vertices as
$\G$ and in which two vertices are adjacent if and only if they are at distance $i$ or $j$ in $\G$.
In the recent works of Brouwer and Fiol  \cite{bf14,f16}, it was studied the situation in which the distance-$d$ graph $\G_d$ of $\G$ (or the Kneser graph $K$ of $\G$) with adjacency matrix $\A_d(=p_d(\A))$, where $p_d$ is the distance-$d$ polynomial, has fewer distinct eigenvalues than $\G$. Examples are the so-called half antipodal ($K$ with only one negative eigenvalue, up to multiplicity), and antipodal distance-regular graphs (where $K$ consists of disjoint copies of a complete graph).

Here we study the cases in which $\G$ has few eigenvalues and its distance-$2$ graph $\Gamma_2$ or its distance-$1$-or-$2$ graph $\Gamma_{1,2}$ are strongly regular.
The main result of this paper is a characterization of such (partially) distance-regular graphs, among regular graphs with $d\in \{3,4\}$ distinct nontrivial eigenvalues, in terms of the spectrum and the mean number of vertices at maximal distance $d$ from every vertex. This can be seen as another version of the so-called spectral excess theorem. Other related characterizations of some of these cases were given by Fiol in \cite{f97,f00,f01}.
For background on distance-regular graphs and strongly regular graphs, we refer the reader to Brouwer, Cohen, and Neumaier \cite{bcn89}, Brouwer and Haemers \cite{bh12}, and Van Dam, Koolen and Tanaka \cite{vdkt16}.

Let $\G$ be a regular (connected) graph with degree $k$, $n$ vertices, and spectrum $\spec \G=\{\theta_0^{m_0},\theta_1^{m_1},\ldots,\theta_d^{m_d}\}$, where $\theta_0(=k)>\theta_1>\cdots>\theta_d$, and $m_0=1$.
In this work, we use the following scalar product on the $(d+1)$-dimensional vector space of real polynomials modulo $m(x)=\prod_{i=0}^d (x-\theta_i)$, that is, the minimal polynomial of $\A$.
\begin{equation}
\label{prod}
\langle p,q\rangle_{\G}=\frac{1}{n}\tr (p(\A)q(\A))= \frac{1}{n}\sum_{i=0}^d m_i p(\theta_i)q(\theta_i), \qquad p,q\in \Re_{d}[x]/(m(x)).
\end{equation}
This is a special case of the inner product of symmetric $n\times n$ real matrices $\M$ and $\N$, defined by $\langle \M,\N\rangle=\frac{1}{n}\tr(\M\N)$.
The {\it predistance polynomials} $p_0,p_1,\ldots, p_d$, introduced by Fiol and Garriga \cite{fg97}, are a sequence of orthogonal polynomials with respect to the  inner product \eqref{prod}, normalized in such a way that $\|p_i\|^2_{\G}=p_i(k)$ (this makes sense since it is known that $p_i(k)>0$ for any $i=0,\ldots,d$, see for instance Szeg\"o \cite{s75}).
As every sequence of orthogonal polynomials, the predistance
polynomials satisfy a three-term recurrence of the form
\begin{equation*}
\label{recur-pol}
xp_i=\beta_{i-1}p_{i-1}+\alpha_i p_i+\gamma_{i+1}p_{i+1}\qquad i=0,1,\dots,d,
\end{equation*}
where the constants $\beta_{i-1}$, $\alpha_i$, and
$\gamma_{i+1}$ are called \emph{preintersection numbers} and are the Fourier coefficients of $xp_i$ in terms
of $p_{i-1}$, $p_i$, and $p_{i+1}$, respectively (and
$\beta_{-1}=\gamma_{d+1}=0$), beginning with $p_0=1$ and
$p_1=x$.

Some basic properties of the predistance polynomials and preintersection numbers are included in the following result (see C\'amara, F\`abrega, Fiol, and Garriga \cite{cffg09}, and Diego, F\`abrega, and Fiol \cite{dff16}).

\begin{lema}
\label{ortho-pol}
Let $G$ be a $k$-regular graph with $d+1$ distinct eigenvalues and predistance polynomials $p_0,\ldots,p_d$. Given an integer $\ell\ge 0$, let  $\overline{{\cal C}}_{\ell}$ be the average number of circuits of length $\ell$ rooted at every vertex, that is,
$\overline{{\cal C}}_{\ell}=\frac{1}{n}\sum_{i=0}^d m_i\theta_i^{\ell}$. Then,
\begin{itemize}
\item[$(i)$]
$p_0(x)=1$, $p_1(x)=x$, $p_2(x)=\frac{1}{\gamma_2}(x^2-\alpha_1 x-k)$.
\item[$(ii)$]
For $i=0,\ldots,d$, the two highest terms of the predistance polynomial $p_i$ are given by
$$
\textstyle
p_i(x)=\frac{1}{\gamma_1\cdots\gamma_i}[x^i - (\alpha_1+\cdots+\alpha_{i-1})x^{i-1}+\cdots].
$$
\item[$(iii)$]
$\alpha_i+\beta_i+\gamma_i=k$, for $i=0,\ldots,d$.
\item[$(iv)$] $\alpha_0=0$, $\beta_0=k$, $\gamma_1=1$, $\alpha_1=\overline{{\cal C}}_{3}/\overline{{\cal C}}_{2}$, and
\begin{equation}
\label{gamma2}
    \gamma_2=\frac{\overline{{\cal C}}_{3}^2-\overline{{\cal C}}_{4}k+k^3}{k(\overline{{\cal C}}_{3}+k-k^2)}.
 \end{equation}
\item[$(v)$]
$p_0+p_1+\cdots+p_d=H$, where $H$ is the Hoffman polynomial \cite{hof63}.
\item[$(vi)$]
For every $i=0,\ldots,d$, (any multiple of) the sum polynomial $q_i=p_0+\cdots+p_i$ maximizes the quotient $r(\theta_0)/\|r\|_{\G}$ among the polynomials $r\in \Re_{i}[x]$ (notice that $q_i(\theta_0)^2/\|q_i\|_{\G}^2=q_i(\theta_0$)), and
$$
(1=)q_0(\theta_0)<q_1(\theta_0)<\cdots <q_d(\theta_0)(=H(\theta_0)=n).
$$
\end{itemize}
\end{lema}

A graph $G$ with diameter $D$ is called {\em $m$-partially distance-regular}, for some $m=0,\dots, D$, if its predistance polynomials satisfy $p_i(\A)=\A_i$ for every $i\le m$. In particular,  every $m$-partially distance-regular with $m\ge 1$ must be regular (see Abiad, Van Dam, and Fiol \cite{adf16}). As an alternative characterization, a graph $G$ is $m$-partially distance-regular when the intersection numbers $c_i$ ($i \leq m$), $a_i$ ($i \leq m-1$), $b_i$ ($i \leq m-1$) are well-defined.
In this case, these intersection numbers are equal to the corresponding preintersection numbers $\gamma_i$ ($i \leq m$), $\alpha_i$ ($i \leq m-1$), $\beta_i$ ($i \leq m-1$), and also $k_i$ is well-defined and equal to $p_i(\theta_0)$ for $i \le m$.  We refer to Dalf\'{o}, Van Dam,  Fiol, Garriga, and Gorissen \cite{ddfgg11} for more background.

Then, with this definition, a graph $\G$ with diameter $D=d$ is distance-regular if and only it is $d$-partially distance-regular. In fact, in this case we have the following strongest proposition, which is a combination of results in Fiol, Garriga and Yebra \cite{fgy1b}, and Dalf\'{o}, Van Dam, Fiol, Garriga and Gorissen \cite{ddfgg11}.
\begin{propo}
\label{propo-p(A)=Ad}
A regular graph $\G$ with $d+1$ different eigenvalues (and, hence, with diameter $D\le d$) is distance-regular if and only if there exists a polynomial $p$ of degree $d$ such that $p(\A)=\A_d$, in which case $p=p_d$.\qqed
\end{propo}

\begin{lema}
\label{lemaSPET}
Let $\G$ be a regular graph with diameter $D$, and let $m\le D$ be a positive integer. Let $n_i(u)$ be the number of vertices at distance at most $i\le D$ from vertex $u$ in $\G$, and let $\overline{n_i}=\frac{1}{n}\sum_{u\in V} n_i(u)$ be the average of these numbers of vertices for all $u\in V$. Then, for any nonzero polynomial $r\in \Re_{i}[x]$ we have
\begin{equation}
\label{basic-ineq}
\frac{r(\theta_0)^2}{\|r\|_{\G}^2}\le \overline{n_{i}},
\end{equation}
with equality if and only if  $r$ is a multiple of $q_{i}=p_0+\cdots +p_i$, and $q_i(\A)=\A_0+\cdots +\A_i$.
\end{lema}

\begin{proof}
Let $\S_{i}=\A_0+\cdots +\A_{i}$.
As $\deg r\le i$, we have $\langle r(\A),\J\rangle = \langle r(\A), \S_{i}\rangle$, where $\J$ is the all-one matrix.
But $\langle r(\A),\J\rangle=\langle r,H\rangle_{\G}=r(\theta_0)$. Then, the Cauchy-Schwarz inequality gives
$$
r^2(\theta_0)\le \|r(\A)\|^2\|\S_{i}\|^2=\|r\|^2_{\G}\overline{n_{i}},
$$
whence \eqref{basic-ineq} follows.
Besides, in case of equality we have that $r$ is multiple of $q_i$, by Lemma \ref{ortho-pol}$(vi)$, with $q_i(\theta_0)=\overline{n_{i}}$. Therefore, $q_i(\A)=\alpha \S_{i}$ for some nonzero constant $\alpha$ and taking norms we conclude that $\alpha=1$.
\end{proof}

In fact, as it was shown in Fiol \cite{f02}, the above result still holds if we change the arithmetic mean of the numbers $n_{i}(u)$, $u\in V$, by its harmonic mean.

As a consequence of Lemma \ref{lemaSPET} and Proposition \ref{propo-p(A)=Ad}, we have the  following generalization of the spectral excess theorem, due to Fiol and Garriga \cite{fg97} (for short proofs, see Van Dam \cite{vd08}, and Fiol, Gago and Garriga \cite{fgg10}).

\begin{theo}
\label{basic-theo}
Let $\G$ be a regular graph with $d+1$ distinct eigenvalues $\theta_0>\cdots >\theta_d$, and diameter $D=d$. Let $m\le D$ be a positive integer.
\begin{itemize}
\item[$(i)$]
If $\G$ is $(m-1)$-partially distance-regular for some $m<d$, and $q_{m}(\theta_0)=\overline{n_{m}}$, then $\G$ is $m$-partially distance-regular.
\item[$(ii)$]
If $q_{d-1}(\theta_0)=\overline{n_{d-1}}$, then $\G$ is distance-regular.
\end{itemize}
\end{theo}

\section{The case of distance-regular graphs}

Here we study the case when $\G$ is a distance-regular graph with diameter three or four. In fact, in the first case everything is basically known (see Brouwer \cite{b84}), although only a combinatorial characterization was provided, whereas we think that the spectral characterization is also important. Indeed, Brouwer \cite{b84} proved
the following (see also Proposition 4.2.17$(i)$ in Brouwer, Cohen, and Neumaier \cite{bcn89}):

\begin{propo}\cite{b84}
Let $\G$ be a distance-regular graph with degree $k$ and diameter $d=3$. Then,
\begin{itemize}
\item[$(i)$]
$\Gamma_2$ is strongly regular  $\iff$ $c_3(a_3+a_2-a_1)=b_1a_2$.
\item[$(ii)$]
$\Gamma_{1,2}$ is strongly regular $\iff$ $\G$ has eigenvalue $-1$ $\iff$ $k=b_2+c_3-1$.
\end{itemize}
\end{propo}

Notice that, in this case, $\Gamma_{1,2}$ is strongly regular if and only if its complement $\Gamma_{3}$ is. As commented in the Introduction, the last case was studied for general diameter by Brouwer and Fiol \cite{bf14} and Fiol \cite{f16}.

\begin{propo}
\label{propo-drg}
Let $\Gamma$ be a distance-regular graph with diameter $D=d=3$, and eigenvalues $\theta_0 (= k) > \theta_1 > \theta_2 > \theta_3$.
\begin{itemize}
\item[$(i)$]
The distance-$2$ graph $\Gamma_2$ is strongly regular if and only if $a_2-c_3$ is an eigenvalue of $\Gamma$.
\item[$(ii)$]
 The distance-$1$-or-$2$ graph $\Gamma_{1,2}$ is strongly regular if and only if $a_3-b_2$ is an eigenvalue of $\Gamma$.
\end{itemize}
\end{propo}

\begin{proof}
We only prove $(i)$, as the proof of $(ii)$ is similar.
As $\G_2$ has adjacency matrix $\A_2=p_2(\A)$, where $p_2(x)=\frac{1}{c_2}(x^2 - a_1 x - k)$, it has eigenvalues $\frac{1}{c_2}(\theta^2 - a_1 \theta - k)$, where $\theta$ is an eigenvalue of $\Gamma$.
For two non-trivial eigenvalues $\eta,\theta$ of $\Gamma$, assume that $\eta^2 - a_1 \eta -k = \theta^2 - a_1 \theta - k$.
This implies $\theta = \eta$ or $\theta + \eta = a_1$.
Let $\tau$ be the third non-trivial eigenvalue of $\Gamma$.
Then $k + \theta + \eta + \tau = a_1 + a_2 + a_3$ and the result follows.
The other direction is trivial to see.
\end{proof}

For diameter $D=d=4$ only the case of the distance-$1$-or-$2$ graph is known (see Proposition 4.2.18 in Brouwer, Cohen, and Neumaier \cite{bcn89}). In the following result, we give an equivalent characterization of this case and, moreover, we study the case of the distance-$2$ graph which, as far as we know, it is new.
The proof is as in Proposition \ref{propo-drg}. For instance, notice that, in case $(i)$, for $\Gamma_2$ to have only two nontrivial distinct eigenvalues, the only possibility is that $p_2(\theta_1)=p_2(\theta_4)$ and
$p_2(\theta_2)=p_2(\theta_3)$.

\begin{propo}
\label{propo-drg-G12}
Let $\Gamma$ be a distance-regular graph with diameter four and eigenvalues $\theta_0 (= k) > \theta_1 > \theta_2 > \theta_3> \theta_4$.
\begin{itemize}
\item[$(i)$]
The distance-$2$ graph $\Gamma_2$ is strongly regular if and only if $\theta_1 + \theta_4 = a_1 = \theta_2 + \theta_3$.
\item[$(ii)$]
The distance-$1$-or-$2$ graph $\Gamma_2$ is strongly regular if and only if $\theta_1 + \theta_4 = a_1-c_2 = \theta_2 + \theta_3$.
\end{itemize}
\end{propo}

An example of distance-regular graph satisfying the conditions of Proposition \ref{propo-drg-G12} is the Hamming graph $H(4,3)$ (see Example \ref{Hamming} in the next section).
Another example  would be the (possible) graph corresponding to the feasible array $\{39,32,20,2;1,4,16,30\}$ (see Brouwer, Cohen, and Neumaier \cite[p. 420]{bcn89}). If it exists, this would be a graph with $n=768$ vertices and spectrum
$39^1,15^{52},7^{117},-1^{468},-9^{130}$. In this case, its distance-$2$ graph would have spectrum $312^1$, $24^{182}$, $-8^{585}$.

\section{The case of regular graphs}

Now we want to conclude the same result as above but only requiring that
the graph $\G$ is regular.
In this case, we use the predistance polynomials and preintersection numbers. Notice that  now  $p_i(\A)$ is not necessarily the distance-$i$ matrix $\A_i$
(usually not even a $0$-$1$ matrix).
 However, as above, we consider that $p_2(\A)$ has only three distinct eigenvalues.

\subsection{The case of diameter three}
We begin with the case of $d=3$ (that is, assuming that $\G$ has four distinct eigenvalues).

\begin{theo}
\label{teo(d=3)G2}
Let $\G$ be a regular graph with degree $k$, $n$ vertices, spectrum $\spec \G=\{\theta_0,\theta_1^{m_1},$ $\theta_2^{m_2},\theta_3^{m_3}\}$, where $\theta_0(=k)>\theta_1>\theta_2>\theta_3$, and preintersection number $\gamma_2$ given by \eqref{gamma2}. Let $\overline{k_3}=\frac{1}{n}\sum_{u\in V} k_3(u)$ be the average number of vertices at distance $3$ from every vertex in $\G$.
Consider the polynomials
\begin{align}
\label{s1(x)}
s_1(x) &=x^2-(\theta_1+\theta_3-\gamma_2)x+\gamma_2+\theta_2(\theta_1-\theta_2+\theta_3),\\
\label{s2(x)}
s_2(x) &=x^2-(\theta_1+\theta_2-\gamma_2)x+\gamma_2+\theta_3(\theta_1-\theta_3+\theta_2).
\end{align}
Then,
\begin{equation}
\label{spet-3}
\overline{k_3}\le \displaystyle\frac{n\sum_{i=1}^3m_i(s_j(\theta_i)-\tau_j)^2}{\sum_{i=0}^3m_i(s_j(\theta_i)-\tau_j)^2},
\quad j=1,2,
\end{equation}
where
\begin{equation}
\tau_j=\displaystyle\frac{s_j(\theta_0)\sum_{i=1}^3 m_i s_j(\theta_i)-\sum_{i=1}^3 m_i s_j(\theta_i)^2}
{s_j(\theta_0)(n-1)-\sum_{i=1}^3 m_i s_j(\theta_i)}, \quad j=1,2.
\label{tau1}
\end{equation}
Equality in \eqref{spet-3} holds for some $j\in\{1,2\}$ if and only if $\G$ is a distance-regular graph and
its distance-$2$ graph $\Gamma_2$
is strongly regular, with eigenvalues
$$
\lambda_0=n-\overline{k_3}-\theta_0-1,\quad
\lambda_1=((\theta_1-\theta_2)(\theta_2-\theta_3)-\tau_1)/\gamma_2, \mbox{ and }  \lambda_2=-\tau_1/\gamma_2,
$$
or
$$
\lambda_0=n-\overline{k_3}-\theta_0-1,\quad
\lambda_1=-\tau_2/\gamma_2, \mbox{ and }  \lambda_2=((\theta_1-\theta_3)(\theta_3-\theta_2)-\tau_2)/\gamma_2,
$$
where $k_3=\overline{k_3}$ is the constant value of the number of vertices at distance $3$
from any vertex in $\G$.
\end{theo}

\begin{proof}
Taking into account that the eigenvalues of $p_2(\A)$ interlace those of $\G$
(because of the orthogonality of the predistance polynomials with respect to the scalar product in \eqref{prod}, see for instance
C\'amara, F\`abrega, Fiol, and Garriga~\cite{cffg09}, or Szeg\"o~\cite{s75}), the only possible cases are:
\begin{enumerate}
\item
$p_2(\theta_1)=p_2(\theta_3)=\sigma_1/\gamma_2$ and $p_2(\theta_2)=-\tau_1/\gamma_2$,
\item
$p_2(\theta_1)=p_2(\theta_2)=\sigma_2/\gamma_2$ and $p_2(\theta_3)=-\tau_2/\gamma_2$,
\end{enumerate}
where $\sigma_j$ and $\tau_j$, for $j=1,2$, are constants.
We only prove the first case, as the other is similar. The main idea is to apply Lemma \ref{lemaSPET} with a polynomial $r\in \Re_{2}[x]$
having the desired properties of (any multiple of) $q_{2}$. To this end, let us assume
that $p_2(\theta_1)=p_2(\theta_3)=\sigma_1/\gamma_2$, and $p_2(\theta_2)=-\tau_1/\gamma_2$
where $\sigma_1$ and $\tau_1$ are constants.
Thus, if we consider a generic monic polynomial $r(x)=x^2+\alpha x+\beta=\gamma_2 q_2(x)$,
where $q_2(x)=p_2(x)+x+1$, we must have
\begin{align*}
r(\theta_1) &= \theta_1^2+\alpha\theta_1+\beta =\sigma_1+\gamma_2\theta_1+\gamma_2,\\
r(\theta_2) &= \theta_2^2+\alpha\theta_2+\beta =-\tau_1+\gamma_2\theta_2+\gamma_2,\\
r(\theta_3) &= \theta_3^2+\alpha\theta_3+\beta =\sigma_1+\gamma_2\theta_3+\gamma_2.
\end{align*}
From the first and last equation we get $\alpha=\gamma_2-\theta_1-\theta_3$ and, hence, the second equation yields $\beta=\gamma_2+\theta_2(\theta_1-\theta_2+\theta_3)-\tau_1$.
Then, we must take $r(x)=s_1(x)-\tau_1$, where $s_1(x)$ is as in \eqref{s1(x)},
and \eqref{basic-ineq} yields
\begin{equation}
\label{Phi(tau)}
\Phi(\tau_1)=\frac{r(\theta_0)^2}{\|r\|_{\G}^2}=\frac{n(s(\theta_0)-\tau_1)^2}
{\sum_{i=0}^3 m_i(s(\theta_i)-\tau_1)^2}\le \overline{s_2}=n-\overline{k_3}.
\end{equation}

Now, to have the best result in \eqref{Phi(tau)}, and, since we are mostly interested in the case of equality, we find the maximum of the function $\Phi$, which is attained at
$\tau_1$ given by \eqref{tau1}. Then, as $\overline{s_2}=n-\overline{k_3}$, the claimed inequality follows. Moreover,
in case of equality, we know, by Theorem \ref{basic-theo}, that $\G$ is distance-regular with $r(x)=\gamma q_{d-1}(x)$ for some constant $\gamma$, which it is 
$\gamma=\gamma_2$.
Then, we get (with standard notation $P_{ij}=p_j(\theta_i)$)
\begin{align*}
P_{22} &= p_2(\theta_2)=-\frac{\tau_1}{\gamma_2},\\
P_{i2} &= p_2(\theta_i)=\frac{\sigma_1}{\gamma_2}=\frac{1}{\gamma2}((\theta_1-\theta_2)(\theta_2-\theta_3)-\tau_1),  \qquad i=1,3.
\end{align*}
To prove the converse, we only need to carry out a simple computation. Indeed, assume that $\G$ is a distance-regular graph, with $k_i$ being the vertices at distance $i=1,2,3$ from any vertex ($k_1=k$), and $p_2(\theta_1)=p_d(\theta_3)$. Then, the same reasoning as in Proposition \ref{propo-drg} gives $a_1=\theta_1+\theta_3$. Then, from $kb_1=c_2k_2=c_2(n-k_3-k-1)$ and $a_1+b_1+1=k$, we get that $c_2=\frac{k(k-1-\theta_1-\theta_3)}{n-k_3-k-1}$. Thus, by putting $\gamma_2=c_2$ in $s_1(x)$ of \eqref{s1(x)} to compute $\tau_1$ in \eqref{tau1}, the inequality \eqref{spet-3} becomes an equality (since $\overline{k_3}=k_3$).
\end{proof}

The following result gives similar conditions for $\G$ to be distance-regular with the distance-$1$-or-$2$ graph $\G_{1,2}$ being strongly regular.

\begin{theo}
\label{teo(d=3)G12}
Let $\G$ be a regular graph with degree $k$, $n$ vertices, spectrum $\spec \G=\{\theta_0,\theta_1^{m_1},$ $\theta_2^{m_2},\theta_3^{m_3}\}$, where $\theta_0(=k)>\theta_1>\theta_2>\theta_3$, and preintersection number $\gamma_2$. Let $\overline{k_3}=\frac{1}{n}\sum_{u\in V} k_3(u)$ be the average number of vertices at distance $3$ from every vertex in $\G$.
Consider the polynomials
\begin{align}
\label{s1(x)G12}
s_1(x) &=x^2-(\theta_1+\theta_3)x+\gamma_2+\theta_2(\theta_1-\theta_2+\theta_3),\\
\label{s2(x)G12}
s_2(x) &=x^2-(\theta_1+\theta_2)x+\gamma_2+\theta_3(\theta_1-\theta_3+\theta_2).
\end{align}
Then,
\begin{equation}
\label{spet-3G12}
\overline{k_3}\le \frac{n\sum_{i=1}^3m_i(s_j(\theta_i)-\tau_j)^2}{\sum_{i=0}^3m_i(s_j(\theta_i)-\tau_j)^2},
\quad j=1,2,
\end{equation}
where
\begin{equation}
\label{tau}
\tau_j=\frac{s_j(\theta_0)\sum_{i=1}^3 m_i s_j(\theta_i)-\sum_{i=1}^3 m_i s_j(\theta_i)^2}
{s_j(\theta_0)(n-1)-\sum_{i=1}^3 m_i s_j(\theta_i)}, \quad j=1,2.
\end{equation}
Equality in \eqref{spet-3G12} holds with some $j\in\{1,2\}$ if and only if $\G$ is a distance-regular graph and
its distance-$1$-or-$2$ graph $\Gamma_{1,2}$
is strongly regular, 
\end{theo}

\begin{example}
The Odd graph O(4) with 7 points, has $n={7\choose 3}=35=1+4+12+18$ vertices, diameter $d=3$, intersection array $\{4,3,3;1,1,2\}$, and spectrum
$4^1,2^{14},-1^{14},-3^{6}$. Then, the functions $\Phi(\tau_j)$ in \eqref{spet-3}
with $j=1,2$ have maximum values at $\tau_1=18/5$ and $\tau_2=-8$, respectively, and their values are $\Phi(18/5)=138/7$ and $\Phi(-8)=22$. Then, since both numbers are greater than $k_3=18$, its distance-$2$ graph $\G_2$ is not strongly regular.

On the other hand the function $\Phi(\tau_1)$ in \eqref{spet-3G12}
has maximum value at $\tau_1=4$, and $\Phi(4)=18=k_3$.
Hence, its distance-$1$-or-$2$ graph $\G_{1,2}$ (and, hence, also $\Gamma_3$)
is strongly regular with $p_1(x)+p_2(x)=x^2+x-4$, and spectrum $16^1,2^{20},-4^{14}$.
\end{example}

\subsection{The case of diameter four}
The following result deals with the case of $d=4$. As in the case of Theorem \ref{teo(d=3)G12}, we omit this proof as goes along the same lines of reasoning as in Theorem \ref{teo(d=3)G2}.

\begin{theo}
\label{teo(d=4)}
Let $\G$ be a regular graph with degree $k$, $n$ vertices, spectrum $\spec \G=\{\theta_0,\theta_1^{m_1},$ $\theta_2^{m_2},\theta_3^{m_3},\theta_4^{m_4}\}$, where $\theta_0(=k)>\theta_1>\theta_2>\theta_3>\theta_4$, such that $\theta_1+\theta_4=\theta_2+\theta_3$, and preintersection number $\gamma_2$. Let $\overline{n_2}=\frac{1}{n}\sum_{u\in V} n_2(u)$ be the average number $n_2(u)=|N_2(u)|$ of vertices  at distance at most $2$ from every vertex $u$ in $\G$.
Consider the polynomials
\begin{align*}
s_1(x) &=x^2-(\theta_2+\theta_3- \gamma_2)x+\theta_2\theta_3,\\
s_2(x) &=x^2-(\theta_2+\theta_3)x+\gamma_2-\theta_2\theta_3.
\end{align*}
Then,
\begin{equation}
\label{spet-4}
\overline{n_2}\ge \Phi(\tau_j)=\frac{n (s_j(\theta_0)-\tau_j)^2}{\sum_{i=0}^4m_i(s_j(\theta_i)
-\tau_j)^2},
\quad j=1,2,
\end{equation}
where
\begin{equation}
\label{tau}
\tau_j=\frac{s_j(\theta_0)\sum_{i=1}^4 m_i s_j(\theta_i)-\sum_{i=1}^4 m_i s_j(\theta_i)^2}
{s_j(\theta_0)(n-1)-\sum_{i=1}^4 m_i s_j(\theta_i)}, \quad j=1,2.
\end{equation}
Equality in \eqref{spet-4} holds with $j=1$ or $j=2$
if and only if $\G$ is a 2-partially distance-regular graph and its distance-$2$
or distance-$1$-or-$2$ graph, respectively, is strongly regular.
\end{theo}

\begin{example}
\label{Hamming}
The Hamming graph $H(4,3)$, with $n=3^4=81$ vertices and diameter $d=4$, has intersection array $\{8,6,4,2;1,2,3,4\}$, so that $k_4=16$, and spectrum
$8^1,5^{8},2^{24},-1^{32},-4^{16}$. Then, the function $\Phi(\tau_j)$ in \eqref{spet-4}
with $j=1$
has a maximum at $\tau_1=4$, and its value is $\Phi(4)=33=s_{2}$. Then, $P_{14}=P_{44}$ and $P_{24}=P_{34}$. Indeed, its distance-$2$ polynomial is $p_2(x)=\frac{1}{2}(x^2-x-8)$ with values $p_4(8)=24$, $p_4(5)=6$, $p_4(2)=-3$, $p_4(-1)=-3$, and $p_4(-4)=6$. Hence, the distance-$2$ graph $\G_2$ is strongly regular with spectrum $24^1,6^{24},-3^{56}$.
\end{example}

\begin{example}
The Odd graph O(5) with 9 points, has $n={9\choose 4}=126=1+5+20+40+60$ vertices, diameter $d=4$, intersection array $\{5,4,4,3;1,1,2,2\}$, and spectrum
$5^1,3^{27},1^{42},-2^{48},-4^{8}$. Then, the function $\Phi(\tau_j)$ in \eqref{spet-4}
with $j=2$
has a maximum at $\tau_2=3$, and its value is $\Phi(4)=26=s_{2}$. Then, its distance-$1$-or-$2$ polynomial is $p_{1,2}(x)=p_1(x)+p_2(x)=x^2+x-5$ with values $p_{1,2}(5)=25$, $p_{1,2}(3)=7$, $p_{1,2}(1)=-3$, $p_{1,2}(-2)=-3$, and $p_{1,2}(-4)=7$. Hence, the distance-$1$-or-$2$ graph $\G_{1,2}$ is strongly regular with spectrum $25^1,7^{35},-3^{90}$.
\end{example}

%
%
%




\begin{thebibliography}{99}

\bibitem{adf16}
A. Abiad, E. R. van Dam, and M. A. Fiol,
Some spectral and quasi-spectral characterizations of distance-regular graphs,
{\em J. Combin. Theory, Ser. A} {\bf 143} (2016) 1--18.

\bibitem{b84}
A. E. Brouwer,
Distance regular graphs of diameter 3 and strongly regular graphs,
{\em Discrete Math.} {\bf 49} (1984) 101--103.

\bibitem{bcn89}
A. E. Brouwer, A. M. Cohen, and A. Neumaier,
\emph{Distance-Regular Graphs},
Springer-Verlag, Berlin-New York, 1989.

\bibitem{bf14}
A. E. Brouwer and M. A. Fiol,
Distance-regular graphs where the distance-$d$ graph has fewer distinct eigenvalues,
{\em Linear Algebra Appl.} {\bf 480} (2015) 115--126.

\bibitem{bh12}
A. E. Brouwer and W. H. Haemers,
\emph{Spectra of Graphs},
Springer, 2012; available online at \url{http://homepages.cwi.nl/~aeb/math/ipm/}.

\bibitem{cffg09}
M. C\'amara, J. F\`abrega, M. A. Fiol, and E. Garriga,
Some families of orthogonal polynomials of a discrete variable and their applications to graphs and codes,
{\em Electron. J. Combin.} {\bf 16(1)} (2009) \#R83.

\bibitem{ddfgg11}
C. Dalf\'o, E. R. van Dam, M. A. Fiol, E. Garriga, and B. L. Gorissen,
On almost distance-regular graphs,
{\em J. Combin. Theory, Ser. A} {\bf 118} (2011) 1094--1113.

\bibitem{vd08}
E. R. van Dam,
The spectral excess theorem for distance-regular graphs: a global (over)view,
{\em Electron. J. Combin.} {\bf 15(1)} (2008) \#R129.

\bibitem{vdkt16}
E. R. van Dam, J. H. Koolen, and H. Tanaka,
Distance-regular graphs,
{\it Electron. J. Combin.} (2016) \#DS22.

\bibitem{dff16}
V. Diego, J. F\`abrega, and M. A. Fiol,
Equivalent characterizations of the spectra of graphs and applications to measures of distance-regularity,
submitted (2016).

\bibitem{f97}
M. A. Fiol,
An eigenvalue characterization of antipodal distance-regular graphs,
{\em Electron. J. Combin.} {\bf 4} (1997) \#R30.

\bibitem{f00}
M. A. Fiol,
A quasi-spectral characterization of strongly distance-regular graphs,
{\em Electron. J. Combin.} {\bf 7} (2000) \#R51.

\bibitem{f01}
M. A. Fiol,
Some spectral characterization of strongly distance-regular graphs,
{\em Combin. Probab. Comput.} {\bf 10} (2001), no. 2, 127--135.

\bibitem{f02}
M. A. Fiol,
Algebraic characterizations of distance-regular graphs,
\emph{Discrete Math.} {\bf 246} (2002) 111--129.

\bibitem{f16}
M. A. Fiol,
The spectral excess theorem for distance-regular graphs having distance-$d$ graph with fewer distinct eigenvalues,
{\em J. Algebraic Combin.} {\bf 43} (2016), no. 4, 827--836.

\bibitem{fgg10}
M. A. Fiol, S. Gago, and E. Garriga,
A simple proof of the spectral excess theorem for distance-regular graphs,
{\em Linear Algebra Appl.} {\bf 432} (2010) 2418--2422.

\bibitem{fg97}
M. A. Fiol and E. Garriga,
From local adjacency polynomials to locally pseudo-distance-regular graphs,
\emph{J. Combin. Theory Ser. B} {\bf 71} (1997) 162--183.

\bibitem{fgy1b}
M. A. Fiol, E. Garriga, and J. L. A. Yebra,
Locally pseudo-distance-regular graphs,
{\it J.  Combin. Theory Ser. B} {\bf 68} (1996) 179--205.


\bibitem{hof63}
A. J. Hoffman,
On the polynomial of a graph,
{\it Amer. Math. Monthly} {\bf 70} (1963) 30--36.


\bibitem{s75}
G. Szeg\"o,
{\em Orthogonal Polynomials},
4th edition, American Mathematical Society, Providence, R.I., 1975.

\end{thebibliography}
\end{document}